\documentclass[reqno, 11pt]{amsart}
\usepackage{amsmath,mathtools}
 \usepackage{amssymb}
\usepackage{amsthm}
\usepackage{amsmath}
\usepackage{times}
\usepackage{latexsym}
\usepackage[mathscr]{eucal}

\numberwithin{equation}{section}
 
  \newtheorem{theorem}{Theorem}[section]
  
  \newtheorem{lemma}[theorem]{Lemma}
  \newtheorem{corollary}[theorem]{Corollary}

  \newtheorem{remark}[theorem]{Remark}
  
  \newtheorem{definition}[theorem]{Definition}
  \newtheorem{example}[theorem]{Example}

\title[New classes of null hypersurfaces]{New classes of null hypersurfaces in indefinite Sasakian space-forms}
\author[Samuel Ssekajja]{Samuel Ssekajja*}
\newcommand{\acr}{\newline\indent}
\address{\llap{*\,} School of Mathematics\acr
 University of the Witwatersrand\acr
 Private Bag 3, Wits 2050\acr
South Africa}
\email{samuel.ssekajja@wits.ac.za} 
\thanks{}


\begin{document}
\begin{abstract}
We introduce two classes of null hypersurfaces of an indefinite Sasakian manifold, $(\overline{M}, \overline{\phi},\zeta, \eta)$, tangent to the characteristic vector field $\zeta$, called; {\it contact screen conformal} and {\it contact screen umbilic} null hypersurfaces. These hypersurfaces come in to fill the existing gap in screen conformal and screen totally umbilic null hypersurfaces. We prove that such hypersurfaces are contained in indefinite Sasakian space forms of constant $\overline{\phi}$-sectional curvature of $-3$.
\end{abstract}
\maketitle
%%%%%%%%%%%%%%%%%%%%%%%%%%%%%%%%%%%%%%%%%%%
\section{Introduction} 
%%%%%%%%%%%%%%%%%%%%%%%%%%%%%%%%%%%%%%%%%%%

The theory of non-degenerate submanifolds \cite{bla1, bla2, oneil} of Riemannian or semi-Riemannian manifolds is one of the most important topics of differential geometry. But the theory of null submanifolds of semi-Riemannian manifolds is relatively new and in a developing stage, with a lot of applications to mathematical physics (general relativity and electromagnetism). The geometry of null submanifolds becomes more difficult and is completely different from that of non-degenerate submanifolds. Such difficulty stems from a non-trivial intersection of the normal bundle of a null submanifold and its tangent bundle. In 1996, Duggal and Bejancu published their work \cite{db} on null submanifolds of semi-Riemannian manifolds and indefinite Kaehler manifolds. This was later, in 2009, updated to \cite{ds2} by Duggal and Sahin, to include the geometry of indefinite almost contact manifolds. In particular, the geometry of null submanifolds of indefinite Sasakian manifolds is in Chapter 7. 

Recently several authors have studied the geometry of null submanifolds of indefinite Sasakian manifolds \cite{Jin, Jin11, kang, massamba1, massamba2}. In particular, \cite{Jin11} proves that there exist no null hypersurfaces of an indefinite Sasakian space form of constant $\overline{\phi}$-sectional curvature different from $1$ (see Theorem 3.3 and Corollary 1). On the other hand, \cite{kang, massamba1, massamba2} considers totally umbilic, screen totally umbilic, $\eta$-totally umbilic, screen conformal and invariant null hypersurfaces in indefinite Sasakian manifolds. In this paper, we prove that all the above null hypersurfaces are non-existent in indefinite Sasakian manifolds, (see Theorem \ref{main1}), if they are assumed to be tangent to the structure vector field, and that there exist other null hypersurfaces in indefinite Sasakian space forms of constant $\overline{\phi}$-sectional curvature different from $1$ (see Theorems \ref{main2} and \ref{main3}).

The rest of the paper is arranged as follows; in Section \ref{pre} we quote the basic notions required in the rest of the paper. In Section \ref{nonexi}, we prove a non-existence result of null hypersurfaces, tangent to the structure vector field. Finally, in Section \ref{newclass}, we introduce two classes of null hypersurfaces in indefinite Sasakian manifolds.

%%%%%%%%%%%%%%%%%%%%%%%%%%%%%%%%%%%%%%%%%%%%5
\section{Preliminaries} \label{pre}
%%%%%%%%%%%%%%%%%%%%%%%%%%%%%%%%%%%%%%%%%%%%%

A $(2n+1)$-dimensional differentiable manifold $\overline{M}$ is said to have an almost contact structure $(\overline{\phi},\zeta, \eta)$ if it admits a vector field $\zeta$, a 1-form $\eta$ and a field $\overline{\phi}$ of endomorphism of the tangent vector space satisfying
\begin{align}\label{p1}
	\overline{\phi}^{2}=-I+\eta \otimes \zeta,\;\;\;\eta(\zeta)=1.
\end{align}
It follows that $\overline{\phi}\zeta=0$, $\eta \circ \overline{\phi}=0$ and $\mathrm {rank}(\overline{\phi})=2n$. Then, the manifold $\overline{M}$, with a $(\overline{\phi},\zeta, \eta)$-structure is called an {\it almost contact manifold} \cite{bla1}. From the point of view of differential geometry it is desirable to define a metric on a paracompact manifold $\overline{M}$. We say that a semi-Riemannian metric $\overline{g}$ is an associated metric of an almost contact structure $(\overline{\phi},\zeta, \eta)$ of $\overline{M}$ if 
\begin{align}\label{p2}
	\overline{g}(\overline{\phi}X,\overline{\phi}Y)=\overline{g}(X,Y)-\eta(X)\eta(Y),\;\;\;\forall\, X,Y\in \Gamma(T\overline{M}).
\end{align}
Here, and in the rest of this paper, $\Gamma(\Xi)$ denotes the smooth sections of a vector bundle $\Xi$. Then, $(\overline{M},\overline{g})$, with $\overline{g}$ satisfying (\ref{p2}), is called an {\it almost contact metric manifold} with a $(\overline{\phi},\zeta, \eta)$-structure and $\overline{g}$ is called its compatible (or associated) metric, whose fundamental 2-form $\Omega$ is defined by $\Omega(X,Y)=\overline{g}(X,\overline{\phi}Y)$, for any $X,Y\in\Gamma(T\overline{M})$. Note that Since $\Omega$, also satisfies $\eta\wedge \Omega^{n}$. A $(2n+1)$-dimensional manifold $\overline{M}$ is called a contact manifold if there exists a 1-form $\eta$ on $\overline{M}$ such that $\eta\wedge (d\eta^{n})\ne 0$ everywhere. When $\Omega=d\eta$ (i.e., $\Omega$ is closed), then, $(\overline{M},\overline{\phi},\zeta, \eta)$ is called a {\it contact metric manifold} \cite{bla2}. These metrics can be constructed by the polarization of $\Omega=d\eta$ evaluated on a local orthonormal basis of the tangent space with respect to an arbitrary metric, on the contact subbundle $\mathcal{D}$, such that $T\overline{M}=\mathcal{D}\perp \mathbb{R}\zeta$. Here, and in the rest of this paper,  $\mathbb{R}\zeta$ denotes the line bundle spanned by $\zeta$. Thus, a contact metric manifold is an analogue of an almost Kaehler manifold, in odd dimensions. The almost contact structure $(\overline{\phi},\zeta, \eta)$ on $\overline{M}$ is said to be {\it normal} if $\overline{\phi}$ is integrable. It is known (see Blair \cite{bla2} for details) that $\overline{M}$ has a {\it normal contact structure} if $N_{\overline{\phi}}+2d\eta \otimes \zeta=0$, where $N_{\overline{\phi}}=[\overline{\phi},\overline{\phi}]$  is the Nijenhuis tensor field of $\overline{\phi}$. A normal contact metric manifold is called a {\it Sasakian manifold}. It is well-known \cite{bla2}  that an almost contact metric manifold $(\overline{M},\overline{g})$ is Sasakian if and only if 
\begin{align}\label{k1}
	(\overline{\nabla}_{X}\overline{\phi})Y=\overline{g}(X,Y)\zeta-\eta(Y)X,\;\;\;\;\forall\,X,Y\in \Gamma(T\overline{M}),
\end{align}
where $\overline{\nabla}$ is the metric connection on $\overline{M}$. Replacing $Y$ by $\zeta$ in (\ref{k1}), and using (\ref{p1}), we get 
\begin{align}\label{p4}
	\overline{\nabla}_{X}\zeta=-\overline{\phi}X,\;\;\;\,\forall\, X\in \Gamma(T\overline{M}).
\end{align}

A plane section $\pi$ in $T_{x}\overline{M}$ of a Sasakian manifold $\overline{M}$ is called a $\overline{\phi}$-section if it is spanned by a unit vector $X$ orthogonal to $\zeta$ and $\overline{\phi}X$, where $X$ is a non-null vector field on $\overline{M}$. The sectional curvature $K(X,\overline{\phi}X)$ of a $\overline{\phi}$-section is called a $\overline{\phi}$-sectional curvature. If $\overline{M}$ has a $\overline{\phi}$-sectional curvature $c$ which does not depend on the $\overline{\phi}$-section at each point, then, $c$ is constant in $\overline{M}$ and $\overline{M}$ is called a {\it Sasakian space form}, denoted by $\overline{M}(c)$. Moreover, the curvature tensor $\overline{R}$ of $\overline{M}$ satisfies (see \cite{bla2})
\begin{align}\label{p25}
	\overline{R}(X,Y)Z=&\frac{(c+3)}{4}\{ \overline{g}(Y,Z)X-\overline{g}(X,Z)Y\}+\frac{(c-1)}{4}\{ \eta(X)\eta(Z)Y \nonumber\\
	        &-\eta(Y)\eta(Z)X+\overline{g} (X,Z)\eta(Y)\zeta- \overline{g} (Y,Z)\eta(X)\zeta\nonumber\\
	        &+\overline{g}(\overline{\phi}Y,Z)\overline{\phi}X-\overline{g}(\overline{\phi}X,Z)\overline{\phi}Y-2\overline{g}(\overline{\phi}X,Y)\overline{\phi}Z\},
	        \end{align}
for all $X,Y,Z\in \Gamma(T\overline{M}).$

Let $(\overline{M},\overline{g})$ be a $(2n+1)$-dimensional semi-Riemannian manifold with index $q$, where  $0< q < 2n+1$, and let $(M,g)$ be a hypersurface of $\overline{M}$. Let $g$ be the induced tensor field by $\overline{g}$ on $M$. Then, $M$ is called a \textit{null hypersurface} of $\overline{M}$ if $g$ is of constant rank $2n-1$ and the normal bundle $TM^{\perp}$ is a distribution of rank 1 on $M$ \cite{db}. Here,  the fibers the  vector bundle $TM^{\perp}$ are defined as $T_{x}M^{\perp}=\{Y_{x}\in T_{x}\overline{M}:\overline{g}_{x}(X_{x},Y_{x})=0,\;\;\forall\, X_{x}\in T_{x}M\}$, for any $x\in M$.  Let $M$  be a null hypersurface, we consider the complementary distribution $S(TM)$ to $TM^{\perp}$ in $TM$, which is called a \textit{screen distribution}. It is well-known that $S(TM)$ is non-degenerate (see \cite{db}). Thus,  $TM=S(TM)\perp TM^{\perp}$. As $S(TM)$ is non-degenerate with respect to $\overline{g}$, we have $T\overline{M}=S(TM)\perp S(TM)^{\perp}$,
where $S(TM)^{\perp}$ is the complementary vector bundle to $S(TM)$ in $T\overline{M}|_{M}$. Let $(M,g)$ be a null hypersurface of $(\overline{M},\overline{g})$ . Then, there exists a unique vector bundle $tr(TM)$, called the \textit{null transversal bundle} \cite{db} of $M$ with respect to $S(TM)$,  of rank 1 over $M$ such that for any non-zero section $E$ of $TM^{\perp}$ on a coordinate neighborhood $\mathcal{U}\subset M$, there exists a unique section $N$ of $tr(TM)$ on $\mathcal{U}$ satisfying 
\begin{align}\label{p5}
	\overline{g}(\xi,N)=1,\;\overline{g}(N,N)=\overline{g}(N,Z)=0,\;\;\;\forall\, Z\in \Gamma(S(TM)).
\end{align}
 Consequently, we have the following decomposition of $T\overline{M}$.  
\begin{align}\nonumber
	T\overline{M}|_{M}=S(TM)\perp \{TM^{\perp}\oplus tr(TM)\}=TM\oplus tr(TM).
\end{align}
  Let $\nabla$ and $\nabla^{*}$ denote the induced connections on $M$ and $S(TM)$, respectively, and $P$ be the projection of $TM$ onto $S(TM)$, then the local Gauss-Weingarten equations of $M$ and $S(TM)$ are the following \cite{db}.
\begin{align}
 &\overline{\nabla}_{X}Y=\nabla_{X}Y+h(X,Y)=\nabla_{X}Y+B(X,Y)N,\label{int1}\\
 &\overline{\nabla}_{X}N=-A_{N}X+\nabla_{X}^{t}N=-A_{N}X+\tau(X)N,\label{flow6}\\
  &\nabla_{X}PY=\nabla^{*}_{X}PY+h^{*}(X,PY)= \nabla^{*}_{X}PY + C(X,PY)\xi,\label{int3}\\
  &\nabla_{X}\xi =-A^{*}_{\xi}X+\nabla_{X}^{*t}\xi=-A^{*}_{\xi}X -\tau(X) \xi,\;\;\; A_{\xi}^{*}\xi=0,\label{flow7}
 \end{align}
 for all $X,Y\in\Gamma(TM)$, $\xi\in\Gamma(TM^{\perp})$ and $N\in\Gamma(tr(T M))$, where $\overline{\nabla}$ is the Levi-Civita connection on $\overline{M}$. In the above setting, $B$ is the local second fundamental form of $M$ and $C$  is the local second fundamental form on $S(TM)$.  $A_{N}$ and $A^{*}_{\xi}$ are the shape operators on $TM$ and $S(TM)$ respectively, while $\tau$ is a 1-form on $TM$. The above shape operators are related to their local fundamental forms by 
 \begin{align}\label{p9}
 	g(A^{*}_{\xi}X,Y) =B(X,Y)\;\;\;\mbox{and}\;\;\;g(A_{N}X,PY) = C(X,PY), 
 \end{align}
 for any $X,Y\in \Gamma(TM)$. It follows from (\ref{flow7}) and (\ref{p9}) that 
 \begin{align}\label{p10}
 	B(X,\xi)=0,\;\;\;\;\forall\, X\in \Gamma(TM).
 \end{align}
 Moreover, we have
 \begin{align}\label{p11}
 	\overline{g}(A^{*}_{\xi}X,N)=0\;\;\;\mbox{and}\;\;\;\overline{g}(A_{N} X,N) = 0, 
 \end{align}
 for all $ X\in\Gamma(TM)$. From  relations (\ref{p11}), we notice that $A_{\xi}^{*}$ and $A_{N}$ are both screen-valued operators.  
 
 Let $\vartheta=\overline{g}(N,\boldsymbol{\cdot})$ be a 1-form metrically equivalent to $N$ defined on $\overline{M}$. Take $\theta=i^{*}\vartheta$
to be its restriction on $M$, where $i:M\rightarrow \overline{M}$ is the inclusion map. Then it is easy to show that 
\begin{align}\label{p40}
	(\nabla_{X}g)(Y,Z)=B(X,Y)\theta(Z)+B(X,Z)\theta(Y), 
\end{align}
for all $X,Y,Z\in \Gamma(TM)$. Consequently,  $\nabla$ is generally \textit {not} a metric connection with respect to $g$. However, the induced connection $\nabla^{*}$ on $S(TM)$ is a metric connection. Denote by $\overline{R}$ and  $R$ the curvature tensors of the connection $\overline{\nabla}$ on $\overline{M}$ and the induced linear connections $\nabla$, respectively. Using the Gauss-Weingarten formulae, we obtain the following  Gauss-Codazzi equations for $M$ and $S(TM)$ (see details in \cite{db,ds2}).
\begin{align}
\overline{g}(\overline{R}(X,Y)Z,\xi)=&(\nabla_{X}B)(Y,Z)-(\nabla_{Y}B)(X,Z)+\tau(X)B(Y,Z)\nonumber\\
&-\tau(Y)B(X,Z),\label{v34}\\
\overline{g}(\overline{R}(X,Y)PZ,N)=&(\nabla_{X}C)(Y,PZ)-(\nabla_{Y}C)(X,PZ)-\tau(X)C(Y,PZ)\nonumber\\
&+\tau(Y)C(X,PZ),\label{v35}
\end{align}
for all $X,Y,Z\in \Gamma(TM)$, $\xi\in \Gamma(TM^{\perp})$ and $N\in \Gamma(tr(TM))$, where $\nabla B$ and $\nabla C$ are defined as follows;
\begin{align}\label{p100}
	(\nabla_{X}B)(Y,Z)=XB(Y,Z)-B(\nabla_{X}Y,Z)-B(Y,\nabla_{X}Z),
\end{align}
 and 
 \begin{align}\label{p101}
	(\nabla_{X}C)(Y,PZ)=XC(Y,PZ)-C(\nabla_{X}Y,PZ)-C(Y,\nabla^{*}_{X}PZ).
\end{align}	
	A null hypersurface $(M,g)$ of a semi-Riemannian manifold $(\overline{M},\overline{g})$ is called \textit{screen conformal} \cite[p. 51]{db} if there exist a non-vanishing smooth function $\psi$ on a neighborhood $\mathcal{U}$ in $M$ such that $A_{N}=\psi A^{*}_{\xi}$, or equivalently, 
	\begin{align}\label{p6}
		C(X,PY)=\psi B(X,PY),
	\end{align}
 for all $X,Y\in \Gamma(TM)$. We say that $M$ is \textit{screen homothetic} if $\psi$ is a constant function on $M$. The null hypersurface $M$ is said to be \textit{totally umbilic} \cite{db} if 
 \begin{align}\label{p7}
 	 B(X,Y)=\rho g(X,Y),
 \end{align}
where $\rho$ is a smooth function on a coordinate neighborhood  $\mathcal{U}\subset TM$. In case $\rho=0$, we say that $M$ is \textit{totally geodesic}. In the same line, $M$ is called \textit{screen totally umbilic} if 
\begin{align}\label{p8}
C(X,PY)=\varrho g(X,PY),	
\end{align}
where $\varrho$ is a smooth function on a coordinate neighborhood $\mathcal{U}\subset TM$. When $\varrho=0$, we say that $M$ is \textit{screen totally geodesic}. 
	
\section{A nonexistance result}\label{nonexi}
Let $(M,g)$ be a null hypersurface of an indefinite Sasakian manifold $(\overline{M}, \overline{\phi},\zeta, \eta)$, which is {\it tangent} to the structure vector field $\zeta$. That is; $\zeta\in \Gamma(TM)$. In such a case, $\zeta$ belongs to $S(TM)$ \cite{calin}. As seen in the previous section, let $\xi$ and $N$ the metric normal and the transversal sections, respectively. Since $(\overline{\phi},\zeta, \eta))$ is an almost contact structure and $\overline{\phi}\xi$ is a null vector field, it follows that $\overline{\phi}N$ is null too. Moreover, $\overline{g}(\overline{\phi}\xi,\xi)=0$ and, thus, $\overline{\phi}\xi$ is tangent to $TM$. Let us consider $S(TM)$ containing $\overline{\phi}TM^{\perp}$ as a vector subbundle. Consequently, $N$ is orthogonal to $\overline{\phi}\xi$ and we have $\overline{g}(\overline{\phi}N,\xi)=-\overline{g}(N,\overline{\phi}\xi)=0$ and $\overline{g}(\overline{\phi}N,N)=0$. This means that $\overline{\phi}N$ is tangent to $TM$ and in particular, it belongs to $S(TM)$. Thus, $\overline{\phi}tr(TM)$ is also a vector subbundle of $S(TM)$. In view of (\ref{p2}), we have 
$\overline{g}(\overline{\phi}\xi,\overline{\phi}N)=1$.

It is then easy to see that $\overline{\phi}TM^{\perp}\oplus \overline{\phi}tr(TM)$ is a non-degenerate vector subbbundle of $S(TM)$, with 2-dimensional fibers. Since $\zeta$ is tangent to $M$, and that $\overline{g}(\overline{\phi}E,\zeta)=\overline{g}(\overline{\phi}N,\zeta)=0$, then there exists a non-degenerate distribution $D_{0}$ on $TM$ such that 
\begin{align}\label{p13}
	S(TM)=\{\overline{\phi}TM^{\perp}\oplus \overline{\phi}tr(TM)\}\perp D_{0}\perp \mathbb{R}\zeta.
\end{align}
Here, $\mathbb{R}\zeta$ denotes the line bundle spanned by $\zeta$. It is easy to check that $D_{0}$ is an almost complex distribution with respect to $\overline{\phi}$, i.e., $\overline{\phi}D_{0}=D_{0}$. The decomposition of $TM$ becomes 
$TM=\{\overline{\phi}TM^{\perp}\oplus \overline{\phi}tr(TM)\}\perp D_{0}\perp \mathbb{R}\zeta\perp TM^{\perp}$. If we set $D:=TM^{\perp}\perp \overline{\phi}TM^{\perp}\perp D_{0}$ and $D'=\overline{\phi}tr(TM)$, then 
\begin{align}\label{p14}
	TM=D\oplus D'\perp \mathbb{R}\zeta.
\end{align} 
Here, $D$ is an almost complex distribution and $D'$ is carried by $\overline{\phi}$ just into the transversal bundle. Let us set 
\begin{align}\label{f11}
	U:=-\overline{\phi}N\;\;\; \mbox{and}\;\;\;V:=-\overline{\phi}\xi. 
\end{align}
Then, from (\ref{p14}), any $X\in \Gamma(TM)$ is written as $X=RX+QX+\eta(X)\zeta$ and $QX=u(X)U$, where $R$ and $Q$ are the projection morphisms of $TM$ onto $D$ and $D'$, respectively, and $u$ is a differential 1-form defined on $M$ by 
\begin{align}\label{p16}
	u(X)=g(V,X),\;\;\;\forall\,X\in \Gamma(TM).
\end{align}
Applying $\overline{\phi}$ to $X$ and using (\ref{p1}), we get 
\begin{align}\label{p17}
	\overline{\phi}X=\phi X+u(X)N,
\end{align}
where $\phi$ is a (1,1) tensor field defined on $M$ by $\phi X=\overline{\phi}RX$. Furthemore, we have 
\begin{align}\label{p18}
	\phi^{2}X=-I+\eta(X)\zeta+u(X)U,\;\;\;\;u(U)=1,\;\;\;\forall\, X\in \Gamma(TM).
\end{align}
It is easy to show that 
\begin{align}
	g(\phi X,\phi Y)=g(X,Y)\eta(X)\eta(Y)-u(Y)v(X)-u(X)v(Y),
\end{align}
where $v$ is a 1-form locally defined on $M$ by 
\begin{align}\label{p26}
	v(X)=g(U,X),\;\;\;\forall\, X\in \Gamma(TM).
\end{align}
Note that 
\begin{align}\label{p30}
	g(\phi X,Y)=-g(X,\phi Y)-u(X)\theta(Y)-u(Y)\theta(X),
\end{align}
for all $X,Y\in \Gamma(TM)$.
\begin{lemma}\label{lemma2}
	On a null hypersurface $(M,g)$ of an indefinite Sasakian manifold $(\overline{M}(c), \overline{\phi},\zeta, \eta)$, tangent to $\zeta$, the following important relations holds.
\begin{align}
	\nabla_{X}\zeta &=-\phi X,\label{p19}\\
	B(X,\zeta)&=-u(X),\label{p20}\\
	C(X,\zeta)&=-v(X),\label{p21}\\
	B(X,U)&=C(X,V),\label{p22}\\
	(\nabla_{X}\phi)Y&=g(X,Y)\zeta-\eta(Y)X-B(X,Y)U+u(Y)A_{N}X,\label{p23}
\end{align}
for all $X,Y\in \Gamma(TM)$.
\end{lemma}
\begin{proof}
	A proof uses straightforward calculations, while considering (\ref{k1}), (\ref{p4}) and (\ref{int1})--(\ref{flow7}).
\end{proof}

 \noindent We also have the following lemma.
\begin{lemma}\label{lema1}
	The vector fields $U$ and $V$ in (\ref{f11}) satisfies the relations
	\begin{align}
		\nabla_{X}U&=\phi A_{N}X-\theta(X)\zeta+\tau(X)U,\label{p38}\\
		\mbox{and}\;\;\;\;\nabla_{X}V&=\phi A^{*}_{\xi}X-\tau(X)V,\;\;\;\forall\,X\in \Gamma(TM).\label{p39}
	\end{align}
\end{lemma}
\begin{proof}
Letting $Y=U$ in (\ref{p23})	 and using the fact $\phi U=0$, we get 
\begin{align}\label{p51}
	-\phi \nabla_{X}U=v(X)\zeta-B(X,U)U+A_{N}X,
\end{align}
for all $X\in \Gamma(TM)$. Applying $\phi$ to (\ref{p51}) and using the fact $\phi \zeta=0$, we have 
\begin{align}\label{p53}
	-\phi^{2}\nabla_{X}U=\phi A_{N}X.
\end{align}
Then, in view of (\ref{p18}) and (\ref{p53}), we see that 
\begin{align}\label{p54}
	\nabla_{X}U=\phi A_{N}X+\eta(\nabla_{X}U)\zeta+u(\nabla_{X}U)U.
\end{align}
But $\eta(\nabla_{X}U)=g(\nabla_{X}U,\zeta)=-\overline{g}(U,\overline{\nabla}_{X}\zeta)=\overline{g}(U,\overline{\phi}X)=-\theta(X)$, in which we have used (\ref{p4}) and (\ref{f11}). On the other hand, $u(\nabla_{X}U)=g(\nabla_{X}U,V)=\overline{g}(\overline{\nabla}_{X}U,V)=\overline{g}(\overline{\nabla}_{X}\overline{\phi}N,\overline{\phi}\xi)=\overline{g}(\overline{\phi}\,\overline{\nabla}_{X}N,\overline{\phi}\xi)=\overline{g}(\overline{\nabla}_{X}N,\xi)=\tau(X)$, in which we have used (\ref{p2}), (\ref{f11}) and (\ref{flow6}). Thus, putting all this in (\ref{p54}) proves (\ref{p38}) of the lemma. A proof of (\ref{p39}) also follows easily from (\ref{p23}) by putting $Y=V$, which completes the proof.
\end{proof}
\noindent In what follows, we construct a null hypersurface of an indefinite Sasakian manifold, tangent to $\zeta$. To that end, $(\mathbb{R}^{2n+1}_{q},\overline{\phi}_{0},\zeta, \eta, \overline{g})$ will denote the manifold $\mathbb{R}^{2n+1}_{q}$ with its usual Sasakian structure given by 
\begin{align}\nonumber
	&\eta=\frac{1}{2}\left(dz-\sum_{i=1}^{n}y^{i}dx^{i}\right),\;\;\;\; \zeta=2\partial z,\nonumber\\
	\overline{g}= \eta \otimes &\eta+\frac{1}{4}\left(-\sum_{i=1}^{q/2}dx^{i}\otimes dx^{i} +dy^{i}\otimes dy^{i}+\sum_{i=q+1}^{n}dx^{i}\otimes dx^{i} +dy^{i}\otimes dy^{i} \right),\nonumber\\
	\overline{\phi}_{0}&\left ( \sum_{i=1}^{n}(X_{i}\partial x^{i}+Y_{i}\partial y^{i})+Z\partial z\right)=\sum_{i=1}^{n}(Y_{i}\partial x^{i}-X_{i}\partial y^{i})+\sum_{i=1}^{n}Y_{i}y^{i}\partial z,\nonumber
\end{align}
where $(x^{i};y^{i};z)$ are the Cartesian coordinates. The above construction will help in understanding the following example.
\begin{example}\label{example1}
	\rm{
	Let $\overline{M}=(\mathbb{R}_{2}^{7},\overline{g})$	 be a semi-Euclidean space, where $\overline{g}$ is of signature $(-,+,+,-,+,+,+)$ with respect to canonical basis 
	$$
	\{\partial x^{1},\partial x^{2},\partial x^{3},\partial y^{1},\partial y^{2},\partial y^{3},\partial z\}.
	$$  
	Suppose $M$ is a submanifold of $\mathbb{R}_{2}^{7}$ defined by  $x^{1}=y^{3}$. It is easy to see that a local frame of $TM$ is given by 
	\begin{align}
		\xi=2(\partial x^{1}+\partial y^{3}+y^{1}&\partial z),\;\;\;Z_{1}=2(\partial x^{3}-\partial y^{1}+y^{3}\partial z),\nonumber\\
		 Z_{2}=2(\partial x^{2}+y^{2}\partial z),\;\;\; Z_{3}&=2\partial y^{2},\;\;\;\; Z_{4}=\partial x^{3}+\partial y^{1}+y^{3}\partial z,\nonumber\\
		 Z_{6}&=\zeta=2\partial z.\nonumber
	\end{align}
	Hence, $TM^{\perp}=\mathrm{span}\{\xi\}$, $\overline{\phi}_{0}TM^{\perp}=\mathrm{span}\{Z_{1}\}$. Note that $TM^{\perp} \cap \overline{\phi}_{0}TM^{\perp}=\{0\}$.
	Next, $\overline{\phi}_{0}Z_{2}=-Z_{3}$, which implies that $D_{0}=\mathrm{span}\{Z_{2},Z_{3}\}$ is invariant with respect to $\overline{\phi}_{0}$. By a direct calculation we see that $tr(TM)$ is spanned by 
	\begin{align}\nonumber
		N=-\partial x^{1}+\partial y^{3}-y^{1}\partial z,
	\end{align}
	such that $\overline{\phi}_{0}N=Z_{4}$. It is easy to see that the vector fields $U,V$ of (\ref{f11}) are given by $U=-Z_{4}$ and $V=-Z_{1}$. Hence, $M$ is a null hypersurface tangent to $\zeta$.	
}	
\end{example}
\noindent Next, we prove the following non-existence result.
\begin{theorem}\label{main1}
Let $(M,g)$ be a null hypersuface  of an indefinite Sasakian manifold $(\overline{M}, \overline{\phi},\zeta, \eta)$, tangent to the structure vector field $\zeta$. The following are all true;
\begin{enumerate}
	\item $M$ is not totally umbilc or screen totally umbilic.
	\item $M$ is not $\eta$-totally umbilic.
	\item $M$ is not screen conformal.
	\item $M$ is not invariant. 
	\item $S(TM)$ is not parallel.
	\item $\nabla$ is not a metric connection.
	
\end{enumerate}
\end{theorem}
\begin{proof}
	Suppose that $M$ is totally umbilic in $\overline{M}$. Then, by (\ref{p20}) and (\ref{p7}), we have  $\rho g(X,\zeta)=-u(X)$, for any $X\in \Gamma(TM)$. Setting $X=\zeta$ in this relation gives $\rho g(\zeta,\zeta)=0$, or simply $\rho=0$, since $\zeta$ is a spacelike vector field. This means that $M$ is totally geodesic null hypersurface. It then follows from (\ref{p20}) that $u(X)=0$, for all $X\in \Gamma(TM)$. This is a contradiction as $u(U)=1$, which proves the first assertion of (1). Next, let $M$ be screen totally umbilic. Then (\ref{p8}) and (\ref{p21}) implies that $\varrho g(X,\zeta)=-v(X)$, for any $X\in \Gamma(TM)$. With $X=\zeta$, we have $\varrho=0$ and thus, $v(X)=0$. This is a contradiction since $v(V)=1$, which completes part (1). 
	
	In case $M$ is $\eta$-totally umbilic we see, from (4.20) of \cite[p. 351]{massamba1}, that
	\begin{align}\label{p50}
		h(X,Y)=\lambda \{g(X,Y)-\eta(X)\eta(Y)\}N,
	\end{align} 
 for any $X,Y\in \Gamma(TM)$. Putting $Y=\zeta$ in (\ref{p50}), we get $h(X,\zeta)=0$, for any $X\in \Gamma(TM)$. As $h(X,\zeta)=B(X,\zeta)N$, we notice that $B(X,\zeta)=0$. Thus, (\ref{p20}) gives $u(X)=0$, for all $X\in \Gamma(TM)$, which is a contradiction. This proves (2).
 
 Next, let $M$ be screen conformal. Then, in view of (\ref{p6}) and (\ref{p21}), we have $\psi B(X,\zeta)=-v(X)$. Setting $X=V$ in this relation and using (\ref{p20}), we get $v(V)=0$. Clearly, this is  a contradiction to the fact $v(V)=1$ and therefore, $M$ is never screen conformal, proving part (3). 
	
	Now, assume that $M$ is an invariant null hypersurface. Then, by \cite[p. 27]{massamba2}, $\overline{\phi}X\in TM$, for any $X\in \Gamma(TM)$. Note that this is equivalent to $\overline{\phi}X=\phi X$, for all $X\in \Gamma(TM)$. Therefore, relations (\ref{p4}) and (\ref{int1}) gives $B(X,\zeta)=0$, which,  in view of (\ref{p20}), gives $u(X)=0$. This is a contradiction too. Therefore, $M$ is never invariant, which proves (4). 
	
	To prove (5), we note, from (\ref{int3}), that the screen distribution $S(TM)$ is parallel if and only if $C(X,PY)=0$, for all $X,Y\in \Gamma(TM)$. Let $Y=V$ in this relation and then compare with (\ref{p22}) to get $B(X,U)=C(X,V)=0$, for any $X\in \Gamma(TM)$. Hence, letting $X=\zeta$ in the last relation and compare with (\ref{p20}) gives $u(U)=0$, which is a contradiction. Finally, $\nabla$ is a metric connection if and only if $B=0$ by (\ref{p40}). That is; $M$ is totally geodesic. But this is not possible by part (1). This proves (6) and the theorem is proved.
\end{proof}
\begin{remark}
\rm{
	In \cite{kang}, the authors proves that totally umbilic null hypersurfaces, of indefinite Sasakian space forms $(\overline{M}(c), \overline{\phi},\zeta, \eta)$ and tangent to $\zeta$ exists. Moreover, they show that $c=1$ (also see the book \cite[p. 317]{ds2}). But in view of Theorem \ref{main1} above, such hypersurfaces don't exist. In \cite{massamba1}, the author study the geometry of null hypersurfaces of indefinite Sasakian manifolds, tangent to $\zeta$. Within this paper, the concept of $\eta$-totally umbilic is introduced (see pages 352--352), and ``Theorem 4.6'' regarding such hypersurfaces is proved. Also, in \cite{massamba2}, screen conformal invariant null hypersurfaces are treated. But we have seen, in Theorem \ref{main1}, that  $\eta$-totally umbilic null hypersufaces, screen conformal and invariant null hypersurfaces, tangent to $\zeta$, do not exist in indefinite Sasakian space forms. Therefore, in studying all null hypersurfaces of indefinite Sasakian space forms, which are tangent to $\zeta$,  equations (\ref{p19})--(\ref{p22}) must be treated carefully. Furthermore, the 1-forms $u,v$ in (\ref{p16}) and (\ref{p26}) {\it can not identically vanish} on $M$ as assumed in \cite[p. 28]{massamba2}, since their vanishing affects the normalization condition in (\ref{p50}). }
	\end{remark}
\begin{remark}
\rm{
	In ``Theorem 3.3'' and ``Corollary 1'' of \cite[p. 576--577]{Jin11}, the author proves that there exist no null hypersurfaces of an indefinite Sasakian space form $(\overline{M}(c), \overline{\phi},\zeta, \eta)$ if $c\ne 1$. This contradicts the fact that totally contact umbilic null hypersurfaces exists in $(\overline{M}(c), \overline{\phi},\zeta, \eta)$, with $c=-3$ (see \cite{ds2} and \cite{massamba1}). Here is the cause; first note that, the curvature tensor $R$ of a null submanifold is not Riemannian since the induced connection, $\nabla$, on a null hypersurface $(M,g)$ is not a metric connection (see the expression for $\nabla g$ in (\ref{p40})). This means that $R$  does not enjoy all the symmetries exhibited by Riemannian curvature tensors. In fact, it is easy to show that  $g(R(X,Y)Z,W)=-g(R(Y,X)Z,W)$, but $g(R(X,Y)Z,W)\ne -g(R(X,Y)W,Z)$, for any $X,Y,Z,W\in \Gamma(TM)$. For any null hypersurface $M$, the relation $\overline{R}(X,Y)\xi=R(X,Y)\xi$, for any $X,Y\in \Gamma(TM)$ and $\xi \in \Gamma(TM^{\perp})$, holds (see (3.6) of  \cite[p. 94]{db}). Consequently, 
	\begin{align}\label{p112}
		\overline{g}(\overline{R}(X,Y)\xi,Z)=g(R(X,Y)\xi,Z)\ne  -g(R(X,Y)Z,\xi)=0.
	\end{align}
	 It was assumed in \cite[p. 576]{Jin11} that all signs in (\ref{p112}) are equalities, which  lead to $c=1$, and any further deductions, such  as Corollary 1. For the above reasons, the curvature tensor of a null submanifold must be treated carefully.} 
\end{remark}

\noindent Theorem \ref{main1} shows that most well-known null hypersurfaces do not exist in indefinite Sasakian manifolds. In particular $B$ and $C$ can not be linked by a non-zero conformal factor on the entire null hypersurface $M$. However, such a conformality can be defined partially on $M$. This is the aim of the next section.

\section{Contact screen conformal null hypersurfaces}\label{newclass}
In this section, we introduce two new classes of null hypersurface of indefinite almost contact manifolds $(\overline{M}(c), \overline{\phi},\zeta, \eta)$, tangent to $\zeta$, called; {\it contact screen conformal} and {\it contact screen umbilic}. This is motivated by Theorem \ref{main1}. In particular, the absence of the classical screen conformal and umbilic null hypersurfaces. As this fails mainly in portions of $TM$ containing $\zeta$, we restrict to a region on $TM$ without $\zeta$, and relate $B$ and $C$ via a non-zero conformal factor for the contact screen conformal null hypersurfaces as shown below. From the decomposition (\ref{p14}), we have $TM=D\oplus D'\perp \mathbb{R}\zeta$. Let $\tilde{P}$ be the projection morphism of $TM$ onto $D\oplus D'$. Then, any $X\in \Gamma(TM)$ is represented as $X=\tilde{P}X+\eta(X)\zeta$.  Then, as $B(\zeta,\zeta)=0$ by (\ref{p20}), we have 
\begin{align}\label{p60}
	B(\tilde{P}X,\tilde{P}Y)=B(X,Y)+\eta(Y)u(X)+\eta(X)u(Y),
\end{align}
for any $X,Y\in \Gamma(TM)$. Next, since $\phi \zeta=0$, we see, from (\ref{int3}), that $C(\zeta,\zeta)=0$. Hence, using (\ref{p21}), we derive 
\begin{align}\label{p61}
	C(\tilde{P}X,\tilde{P}PY)=C(X,PY)+\eta(Y)v(X)-\eta(X)\omega(Y),
\end{align}
where $\omega$ is a 1-form on $M$ defined as 
\begin{align}\label{p62}
	\omega(X)=C(\zeta,PX),\;\;\;X\in \Gamma(TM).
\end{align}
In view of (\ref{int3}), (\ref{p19}), (\ref{p20}),  (\ref{p22}) and (\ref{p62}), we have 
\begin{align}\label{p65}
	\omega(\zeta)=0\;\;\;\;\;\mbox{and}\;\;\;\;\;\omega(V)=-1.
\end{align}

\noindent Next, we define {\it contact screen conformal} null hypersurfaces.

\begin{definition}
\rm {
Let $(M,g)$ be a null hypersurface of an indefinite Sasakian manifold $(\overline{M}, \overline{\phi},\zeta, \eta)$, tangent to the structure vector field $\zeta$.  Then, $M$ will be called {\it contact screen conformal} null hypersurface if there exists a non-zero smooth function $\varphi$ on a neighborhood $\mathcal{U}\subset M$ such that 
\begin{align}\label{p65}
	C(\tilde{P}X,\tilde{P}Y)=\varphi B(\tilde{P}X,\tilde{P}Y),\;\;\;\;\forall\, X,Y\in \Gamma(TM).
\end{align}
 Equivalently, using (\ref{p60}) and (\ref{p61}), $M$ is contact screen conformal if 
\begin{align}\label{p66}
		C(X,PY)=\varphi \{B&(X,Y)+\eta(Y)u(X)+\eta(X)u(Y)\}\nonumber\\
		&-\eta(Y)v(X)+\eta(X)\omega(Y),
	\end{align}
	for all $X,Y\in \Gamma(TM)$. We say that $M$ is {\it contact screen homothetic}
	 if $\varphi$ is a conatant function.}
\end{definition}
\noindent As an example, we have the following.
\begin{example}
\rm{
	 Consider the null hypersurface in Example \ref{example1}. By (\ref{p4}) and (\ref{int1}), we have $B(X,\zeta)=-u(X)$, for all $X\in \Gamma(TM)$.  By a direct calculation, we have $[\xi, V]=-4\zeta$, and $[\xi,X]=0$, for all other $X\in \Gamma(TM)$. Thus, using Koszul formula \cite[p. 61]{oneil}, we have $\overline{g}(\overline{\nabla}_{X}Y,\xi)=0$, for all $X,Y\in \Gamma(D\oplus D')$	. Hence, $B=0$ on $D\oplus D'$. On the other hand, $[N,U]=-\zeta$ and $[N,X]=0$ for any other $X\in \Gamma(TM)$. Also, using Koszul formula, we have 	$\overline{g}(\overline{\nabla}_{X}Y,N)=0$, for any $X,Y\in \Gamma(D\oplus D')$. Thus, $C=0$ on $D\oplus D'$. Note that, as $[\zeta,X]=0$ for all $X\in \Gamma(TM)$, we have $C(\zeta, X)=C(X,\zeta)=-v(X)$. Clearly, $(M,g)$ is a contact screen conformal null hypersurface with $\varphi$ an arbitrary function.}
\end{example}\label{example2}
\noindent Next, let $(M,g)$ be a null hypersurface of an indefinite Sasakian space-form. Then, using (\ref{p25}), (\ref{v34}), (\ref{v35}), (\ref{p16}) and (\ref{p26}), we have, for any $X,Y,Z\in \Gamma(TM)$,
\begin{align}\label{p27}
	(\nabla_{X}&B)(Y,Z)-(\nabla_{Y}B)(X,Z)=\tau(Y)B(X,Z)-\tau(X)B(Y,Z)\nonumber\\
	&+\frac{(c-1)}{4}\{\overline{g}(\overline{\phi}Y,Z)u(X) -\overline{g}(\overline{\phi}X,Z)u(Y)-2\overline{g}(\overline{\phi}X,Y)u(Z)\},
\end {align}
and also 
\begin{align}
	(\nabla_{X}C)(Y,PZ)-(\nabla_{Y}C)&(X,PZ)=\tau(X)C(Y,PZ)-\tau(Y)C(X,PZ)\nonumber\\
	&+\frac{(c+3)}{4}\{ \overline{g}(Y,PZ)\theta(X)-\overline{g}(X,PZ)\theta(Y)\}\nonumber\\
	&+\frac{(c-1)}{4}\{\eta(X)\eta(PZ)\theta(Y)-\eta(Y)\eta(PZ)\theta(X)\nonumber\\
	&+\overline{g}(\overline{\phi}Y,PZ)v(X)-\overline{g}(\overline{\phi}X,PZ)v(Y)\nonumber\\
	&-2\overline{g}(\overline{\phi}X,Y)v(PZ)\}.\label{p31}
\end{align}
\noindent In what follows, we prove the following result.
\begin{theorem}\label{main2}
	Let $(M,g)$ be a contact screen conformal null hypersurface of an indefinite Sasakian space form $(\overline{M}(c), \overline{\phi},\zeta, \eta)$. Then, $c=-3$ (that is; $\overline{M}$ is a space of constant $\overline{\phi}$-sectional curvature $-3$) and $\omega(U)=C(\zeta,U)=0$. Moreover, if $B(U,V)$ is nonzero, then $\varphi$ satisfies the differential equation $\xi \varphi-2\varphi\tau(\xi)=0$.
\end{theorem}
\begin{proof}
	From (\ref{p66}), (\ref{flow7}), (\ref{p40}), (\ref{p19}) and  (\ref{p20}), we derive 
	\begin{align}\label{p70}
		(\nabla_{X}C)&(Y,PZ)=(X\varphi)\{B(Y,Z)+\eta(Z)u(Y)+\eta(Y)u(Z)\}\nonumber\\
		&+\varphi\{(\nabla_{X}B)(Y,Z)-g(Z,\phi X)u(Y)+B(X,V)\eta(Z)\theta(Y)\nonumber\\
		&+\eta(Z)g(\nabla_{X}V,Y)-u(X)u(Z)\theta(Y)-g(Y,\phi X)u(Z)\nonumber\\
		&+\eta(Y)g(\nabla_{X}V,Z)\}+g(Z,\phi X)v(Y)-B(X,U)\eta(Z)\theta(Y)\nonumber\\
		&-\eta(Z)g(\nabla_{X}U,Y)-u(X)\theta(Y)\omega(Z)-g(Y,\phi X)\omega(Z)\nonumber\\
		&+\eta(Y)X\omega(Z)-\eta(Y)\omega(\nabla^{*}_{X}PZ),\;\;\;\forall\, X,Y,Z\in \Gamma(TM).
	\end{align}
	Interchanging $X$ and $Y$ in (\ref{p70}) and subtracting the two relations gives 
	\begin{align}\label{p73}
		(\nabla_{X}C)&(Y,PZ)-(\nabla_{Y}C)(X,PZ)\nonumber\\
		&=	(X\varphi)\{B(Y,Z)+\eta(Z)u(Y)+\eta(Y)u(Z)\}-(Y\varphi)\{B(X,Z)\nonumber\\
		&+\eta(Z)u(X)+\eta(X)u(Z)\}+\varphi\{(\nabla_{X}B)(Y,Z)-(\nabla_{Y}B)(X,Z)\nonumber\\
		&-g(Z,\phi X)u(Y)+g(Z,\phi Y)u(X)+B(X,V)\theta(Y)\eta(Z)\nonumber\\
		&-B(Y,V)\theta(X)\eta(Z)+\eta(Z)g(\nabla_{X}V,Y)-\eta(Z)g(\nabla_{Y}V,X)\nonumber\\
		&-u(X)\theta(Y)u(Z)+u(Y)\theta(X)u(Z)-g(Y,\phi X)u(Z)\nonumber\\
		&+g(X,\phi Y)u(Z)+\eta(Y)g(\nabla_{X}V,Z)-\eta(X)g(\nabla_{Y}V,Z)\}\nonumber\\
		&+g(Z,\phi X)v(Y)-g(Z,\phi Y)v(X)-B(X,U)\theta(Y)\eta(Z)\nonumber\\
		&+B(Y,U)\theta(X)\eta(Z)+\eta(Z)g(\nabla_{Y}U,X)-\eta(Z)g(\nabla_{X}U,Y)\nonumber\\
		&-u(X)\theta(Y)\omega(Z)+u(Y)\theta(X)\omega(Z)-g(Y,\phi X)\omega(Z)\nonumber\\
		&+g(X,\phi Y)\omega(Z)+\eta(Y)X\omega(Z)-\eta(X)Y\omega(Z)\nonumber\\
		&-\eta(Y)\omega(\nabla^{*}_{X}PZ)+\eta(X)\omega(\nabla_{Y}^{*}PZ).
	\end{align}
	In view of Lemma \ref{lema1} and  (\ref{p30}), we simplify (\ref{p73}) to 
	\begin{align}\label{p74}
		(\nabla_{X}C)&(Y,PZ)-(\nabla_{Y}C)(X,PZ)\nonumber\\
		&=	(X\varphi)\{B(Y,Z)+\eta(Z)u(Y)+\eta(Y)u(Z)\}-(Y\varphi)\{B(X,Z)\nonumber\\
		&+\eta(Z)u(X)+\eta(X)u(Z)\}+\varphi\{(\nabla_{X}B)(Y,Z)-(\nabla_{Y}B)(X,Z)\nonumber\\
		&-g(Z,\phi X)u(Y)+g(Z,\phi Y)u(X)+\eta(Z)B(Y,\phi X)\nonumber\\
		&-\eta(Z)B(X,\phi Y)+\eta(Z)u(X)\tau(Y)-\eta(Z)u(X)\tau(X)\nonumber\\
		&-u(X)\theta(Y)u(Z)+u(Y)\theta(X)u(Z)-g(Y,\phi X)u(Z)\nonumber\\
		&+g(X,\phi Y)u(Z)-\eta(Y)\theta(Z)B(X,V)-\eta(Y)B(X,\phi Z)\nonumber\\
		&-\eta(Y)u(Z)\tau(X)+\eta(X)\theta(Z)B(Y,V)+\eta(X)B(Y,\phi Z)\nonumber\\
		&+\eta(X)u(Z)\tau(Y)\}+g(Z,\phi X)v(Y)-g(Z,\phi Y)v(X)\nonumber\\
		&-B(X,U)\theta(Y)\eta(Z)+B(Y,U)\theta(X)\eta(Z)+\eta(Z)\theta(Y)C(X,V)\nonumber\\
		&-\eta(Z)\theta(X)C(Y,V)+\eta(Z)C(X,P\phi Y)-\eta(Z)C(Y,P\phi X)\nonumber\\
		&+\eta(Z)\eta(Y)\theta(X)-\eta(Z)\eta(X)\theta(Y)+\eta(Z)v(X)\tau(Y)\nonumber\\
		&-\eta(Z)v(Y)\tau(X)-u(X)\theta(Y)\omega(Z)+u(Y)\theta(X)\omega(Z)\nonumber\\
		&-g(Y,\phi X)\omega(Z)+g(X,\phi Y)\omega(Z)+\eta(Y)X\omega(Z)\nonumber\\
		&-\eta(X)Y\omega(Z)-\eta(Y)\omega(\nabla^{*}_{X}PZ)+\eta(X)\omega(\nabla_{Y}^{*}PZ).
	\end{align}
	Then, putting (\ref{p27}), (\ref{p31}), (\ref{p74}) and  (\ref{p66}), we get 
	\begin{align}\label{p75}
		&(X\varphi)\{B(Y,Z)+\eta(Z)u(Y)+\eta(Y)u(Z)\}-(Y\varphi)\{B(X,Z)\nonumber\\
		&+\eta(Z)u(X)+\eta(X)u(Z)\}+\varphi\{2B(X,Z)\tau(Y)-2B(Y,Z)\tau(X)\nonumber\\
		&+ ((c-1)/4)\{ \overline{g}(\overline{\phi}Y,Z)u(X) -\overline{g}(\overline{\phi}X,Z)u(Y)-2\overline{g}(\overline{\phi}X,Y)u(Z)\} \nonumber\\
		&-\eta(Z)u(Y)\tau(X)-2\eta(Y)u(Z)\tau(X)+\eta(Z)u(X)\tau(Y)\nonumber\\
		&+2\eta(X)u(Z)\tau(Y)-g(Z,\phi X)u(Y)+g(Z,\phi Y)u(X)+\eta(Z)B(Y,\phi X)\nonumber\\
		&-\eta(Z)B(X,\phi Y)+\eta(Z)u(X)\tau(Y)-\eta(Z)u(X)\tau(X)\nonumber\\
		&-u(X)\theta(Y)u(Z)+u(Y)\theta(X)u(Z)-g(Y,\phi X)u(Z)\nonumber\\
		&+g(X,\phi Y)u(Z)-\eta(Y)\theta(Z)B(X,V)-\eta(Y)B(X,\phi Z)\nonumber\\
		&+\eta(X)\theta(Z)B(Y,V)+\eta(X)B(Y,\phi Z)+\eta(Z)\theta(Y)B(X,V)\nonumber\\
		&-\eta(X)\theta(X)B(Y,V)+\eta(Z)\eta(Y)\theta(X)+\eta(Z)B(X,\phi Y)\nonumber\\
		&-\eta(Z)B(Y,\phi X)\}+g(Z,\phi X)v(Y)-g(Z,\phi Y)v(X)\nonumber\\
		&-\eta(Y)\theta(Y)\eta(X)-B(X,U)\theta(Y)\eta(Z)+B(Y,U)\theta(X)\eta(Z)\nonumber\\
		&+\eta(Z)\eta(Y)\theta(X)-\eta(Z)\eta(X)\theta(Y)+\eta(Z)v(X)\tau(Y)\nonumber\\
		&-\eta(Z)v(Y)\tau(X)-u(X)\theta(Y)\omega(Z)+u(Y)\theta(X)\omega(Z)\nonumber\\
		&-g(Y,\phi X)\omega(Z)+g(X,\phi Y)\omega(Z)+\eta(Y)X\omega(Z)\nonumber\\
		&-\eta(X)Y\omega(Z)-\eta(Y)\omega(\nabla^{*}_{X}PZ)+\eta(X)\omega(\nabla_{Y}^{*}PZ)\nonumber\\
		&+\eta(Z)\eta(X)\omega(\phi Y)-\eta(Z)\eta(Y)\omega(\phi X)-\eta(Y)\omega(Z)\tau(X)\nonumber\\
		&+\eta(X)\omega(Z)\tau(Y)-((c+3)/4)\{ \overline{g}(Y,PZ)\theta(X)-\overline{g}(X,PZ)\theta(Y)\}\nonumber\\
		&-((c-1)/4)\{\eta(X)\eta(PZ)\theta(Y)-\eta(Y)\eta(PZ)\theta(X)\nonumber\\
		&+\overline{g}(\overline{\phi}Y,PZ)v(X)-\overline{g}(\overline{\phi}X,PZ)v(Y)-2\overline{g}(\overline{\phi}X,Y)v(PZ)\}=0,	
	\end{align}
	for any $X,Y,Z\in \Gamma(TM)$. Next, setting $X=\xi$ in (\ref{p75}), we get 
	\begin{align}\label{p78}
		(\xi\varphi)&\{B(Y,Z)+\eta(Z)u(Y)+\eta(Y)u(Z)\}-\varphi\{2B(Y,Z)\tau(\xi)\nonumber\\
		&+(3(c-1)/4)u(Y)u(Z)-\eta(Z)u(Y)\tau(\xi)-2\eta(Y)u(Z)\tau(\xi)\nonumber\\
		&+3u(Y)u(Z)+\eta(Y)\eta(Z)\}-u(Z)v(Y)+B(Y,U)\eta(Y)\nonumber\\
		&-\eta(Z)v(Y)\tau(\xi)+2u(Y)\omega(Z)+\eta(Y)\xi \omega(Z)\nonumber\\
		&-\eta(Y)\omega(\nabla^{*}_{\xi}PZ)-\eta(Y)\omega(Z)\tau(\xi)-((c+3)/4)g(Y,PZ)\nonumber\\
		&-((c-1)/4)\{u(PZ)v(Y)+2u(Y)v(PZ)\}=0.
	\end{align}
	Putting $Y=V$ and $Z=U$ in (\ref{p78}) gives 
	\begin{align}\label{p79}
		\{\xi \varphi-2\varphi \tau(\xi)\}B(V,U)-\frac{1}{2}(c+3)=0.
	\end{align}
On the other hand, putting $Y=U$ and $Z=V$ in (\ref{p78}) gives 
\begin{align}\label{p80}
	\{\xi \varphi-2\varphi \tau(\xi)\}B(U,V)-\frac{3}{4}(c+3)=0.
\end{align}
Also, putting $Y=Z=U$ in (\ref{p78}), we get 
\begin{align}\label{p81}
	\{\xi \varphi-2\varphi \tau(\xi)\}B(U,U)+\frac{3}{4}\varphi(c+3)+2\omega(U)=0.
\end{align}
Finally, putting $Y=Z=V$	 leads to 
\begin{align}\label{p82}
	\{\xi \varphi-2\varphi \tau(\xi)\}B(V,V)=0.
\end{align} 
	From (\ref{p79}) and (\ref{p80}), we get 
	\begin{align}\label{p83}
		c=-3\;\;\;\;\mbox{and}\;\;\;\;\{\xi \varphi-2\varphi \tau(\xi)\}B(V,U)=0	.\end{align}
 As $M$ is contact screen conformal, (\ref{p22}) and (\ref{p66}) gives $B(U,U)=\varphi B(V,U)$. Therefore, considering this relation in (\ref{p81}) together with (\ref{p83}), we get $\omega(U)=0$. Note that, by (\ref{p22}) and (\ref{p66}), we have $B(V,U)= \varphi B(V,V)$. Therefore, if $B(V,U)\ne 0$, then $B(U,U)\ne 0$ and $B(V,V)\ne 0$, which all implies that $\xi \varphi-2\varphi\tau(\xi)=0$. This completes the proof.
\end{proof}
\noindent As a consequence of Theorem \ref{main2}, we have the following.
\begin{corollary}
	There exist no contact screen conformal null hypersurfaces of an indefinite Sasakian space form $(\overline{M}(c\ne -3), \overline{\phi},\zeta, \eta)$ with $\zeta\in \Gamma(TM)$.
\end{corollary}
\begin{example}
	\rm{
	In view of Example \ref{example2}, both $B$ and $C$ vanishes on $D\oplus D'$. Hence, $(M,g)$ is a contact screen conformal null hypersurface with $\varphi$ arbitrary,  and $\mathbb{R}_{2}^{7}$ is a space of constant $\overline{\phi}$-sectional curvature $-3$. 	
	
	}
\end{example}
\begin{remark}
\rm{
	Geometrically, the proportionality of  $B$ to the metric tensor $g$ on $D\oplus D'$ is equivalent to $(M,g)$ being totally contact umbilic in $\overline{M}$. In fact, let $B=\lambda \otimes g$ on $D\oplus D'$, where $\lambda$ is a smooth function on $\mathcal{U}\subset M$. Then, by (\ref{p60}), we have $\lambda g(\tilde{P}X,\tilde{P}Y)=B(X,Y)+\eta(Y)u(X)+\eta(X)u(Y)$, for any $X,Y\in \Gamma(TM)$. Thus, using $X=\tilde{P}X+\eta(X)\xi$ to this relation, we get
	\begin{align}\label{p87}
		B(X,Y)=\lambda\{g(X,Y)-\eta(X)\eta(Y)\}-\eta(Y)u(X)-\eta(X)u(Y).
	\end{align}
	Clearly, (\ref{p87}) is the defining condition for $(M,g)$ to be totally contact umbilic  \cite{ds2} null hypersurface. It has already been established that for such a null hypersurface in an indefinite Sasakian space form $(\overline{M}(c), \overline{\phi},\zeta, \eta)$, then $c=-3$ (see \cite{massamba1} for details). On the other hand, we know that $C$ is generally not symmetric on $S(TM)$. In fact, by (\ref{int3}), we have $C(X,Y)-C(Y,X)=\theta([X,Y])$, for any $X,Y\in \Gamma(S(TM))$. Thus, $C$ is symmetric if and only if $S(TM)$ is an integrable distribution on $M$. For this reason, if the screen local second fundamental form $C$  is  proportional to $g$ on $D\oplus D'$, we get it symmetric on $D\oplus D'$ but not on $S(TM)$, and therefore, we do not recover the equation for a totally contact screen umbilic null hypersurface (see \cite{massamba1}) for which it is assumed to be symmetric on $S(TM)$. In fact, if $C=\gamma\otimes g$ on $D\otimes D'$, we see, from (\ref{p61}), that 
	\begin{align}\label{p90}
		C(X,PY)=\gamma\{g(X,Y)-\eta(X)\eta(Y)\}-\eta(Y)v(X)+\eta(X)\omega(Y),
	\end{align}
 for any $X,Y\in \Gamma(TM)$. It follows from (\ref{p90}) and (\ref{p21}) that; if $C$ is symmetric on $S(TM)$ then $\omega(Y)=C(\zeta, Y)=C(Y,\zeta)=-v(Y)$, for any $Y\in \Gamma(S(TM))$, and thus the relation of totally contact screen umbilic is recovered. In such a case, it has been proved that $c=-3$ and, in fact, $(M,g)$ is totally contact screen geodesic, i.e., $\gamma=0$ (see \cite{massamba1}). Therefore, in general, the proportionality of  $C$ to $g$ on $D\oplus D'$ does not imply totally contact screen umbilic considered in \cite{massamba1}. Based on the above, we have the following general definition of a {\it contact screen umbilic} null hypersurfaces.}
\end{remark}
\begin{definition}
\rm {
Let $(M,g)$ be a null hypersurface of an indefinite Sasakian manifold $(\overline{M}, \overline{\phi},\zeta, \eta)$, tangent to the structure vector field $\zeta$.  Then, $M$ will be called {\it contact screen umbilic} null hypersurface if there exists a non-zero smooth function $\gamma$ on a neighborhood $\mathcal{U}\subset M$ such that 
\begin{align}\label{p92}
	C(\tilde{P}X,\tilde{P}PY)=\gamma g(\tilde{P}X,\tilde{P}Y),\;\;\;\;\forall\, X,Y\in \Gamma(TM).
\end{align}
 Equivalently, using (\ref{p61}) and (\ref{p62}), $M$ is contact screen umbilic if 
\begin{align}\label{p91}
		C(X,PY)=\gamma\{g(X,Y)-\eta(X)\eta(Y)\}-\eta(Y)v(X)+\eta(X)\omega(Y),
	\end{align}
	for all $X,Y\in \Gamma(TM)$. We say that $(M,g)$ is {\it contact screen geodesic} if $\gamma=0$.
	}
\end{definition}
\noindent Accordingly, we have the following result. 
\begin{theorem}\label{main3}
	Let $(M,g)$ be a contact screen umbilic null hypersurface of an indefinite Sasakian space form $(\overline{M}(c), \overline{\phi},\zeta, \eta)$. Then, $c=-3$ (that is; $\overline{M}$ is a space of constant $\overline{\phi}$-sectional curvature $-3$). Moreover, $\gamma$ satisfies the relations;
	\begin{align}\nonumber
		\gamma^{2}-2\omega(U)=0,\;\;\;\gamma B(V,V)=0\;\;\; \mbox{and}\;\;\;\xi \gamma-\gamma\tau(\xi)=0.
 	\end{align}
	Furthermore, if $C$	 is symmetric on $S(TM)$, then $M$ is contact screen geodesic.
\end{theorem}
\begin{proof}
	Using (\ref{p91}), (\ref{p19}), (\ref{p20}), (\ref{p21}), (\ref{p40}), we derive 
	\begin{align}\label{p103}
		(\nabla_{X}&C)(Y,PZ)-(\nabla_{Y}C)(X,PZ)=(X\gamma)\{g(Y,Z)-\eta(Y)\eta(Z)\}\nonumber\\
		&-(Y\gamma)\{g(X,Z)-\eta(X)\eta(Z)\}+\gamma \{B(X,PZ)\theta(Y)-B(Y,PZ)\theta(X)\nonumber\\
		&+u(X)\theta(Y)\eta(Z)-u(Y)\theta(X)\eta(Z)+g(Y,\phi X)\eta(Z)-g(X,\phi Y)\eta(Z)\nonumber\\
		&+\eta(Y)g(Z,\phi X)-\eta(X)g(Z,\phi Y)\}+g(Z,\phi X)v(Y)-g(Z,\phi Y)v(X)\nonumber\\
		&+\eta(Z)\theta(X)B(Y,U)-\eta(Z)\theta(Y)B(X,U)+\eta(Z)g(\nabla_{Y}U,X)\nonumber\\
		&-\eta(Z)g(\nabla_{X}U,Y)+u(Y)\theta(X)\omega(Z)-u(X)\theta(Y)\omega(Z)\nonumber\\
		&+g(X,\phi Y)\omega(Z)-g(Y,\phi X)\omega(Z)+\eta(Y)X\omega(Z)-\eta(X)Y\omega(Z)\nonumber\\
		&+\eta(X)\omega(\nabla_{Y}PZ)-\eta(Y)\omega(\nabla_{X}PZ),
	\end{align}
	for all $X,Y,Z\in \Gamma(TM)$. Substituting (\ref{p103}) in (\ref{p31}) and using (\ref{p91}), we get 
	\begin{align}\label{p106}
		(X&\gamma)\{g(Y,Z)-\eta(Y)\eta(Z)\}-(Y\gamma)\{g(X,Z)-\eta(X)\eta(Z)\}\nonumber\\
		&+\gamma \{B(X,PZ)\theta(Y)-B(Y,PZ)\theta(X)+g(X,PX)\tau(Y)-g(Y,PZ)\tau(X)\nonumber\\
		&+u(X)\theta(Y)\eta(Z)-u(Y)\theta(X)\eta(Z)+g(Y,\phi X)\eta(Z)-g(X,\phi Y)\eta(Z)\nonumber\\
		&+\eta(Y)g(Z,\phi X)-\eta(X)g(Z,\phi Y)+\eta(Y)\eta(Z)\tau(X)-\eta(X)\eta(Z)\tau(Y)\}\nonumber\\
		&+g(Z,\phi X)v(Y)-g(Z,\phi Y)v(X)+\eta(Z)\theta(X)B(Y,U)\nonumber\\
		&-\eta(Z)\theta(Y)B(X,U)+\eta(Z)g(\nabla_{Y}U,X)-\eta(Z)g(\nabla_{X}U,Y)\nonumber\\
		&+u(Y)\theta(X)\omega(Z)-u(X)\theta(Y)\omega(Z)+g(X,\phi Y)\omega(Z)\nonumber\\
		&-g(Y,\phi X)\omega(Z)+\eta(Y)X\omega(Z)-\eta(X)Y\omega(Z)+\eta(X)\omega(\nabla_{Y}PZ)\nonumber\\
		&-\eta(Y)\omega(\nabla_{X}PZ)+\eta(Z)v(Y)\tau(X)-\eta(Z)v(X)\tau(Y)\nonumber\\
		&+\eta(X)\omega(Z)\tau(Y)-\eta(Y)\omega(Z)\tau(X)-((c+3)/4)\{ \overline{g}(Y,PZ)\theta(X)\nonumber\\
		&-\overline{g}(X,PZ)\theta(Y)\}-((c-1)/4)\{\eta(X)\eta(PZ)\theta(Y)-\eta(Y)\eta(PZ)\theta(X)\nonumber\\
		&+\overline{g}(\overline{\phi}Y,PZ)v(X)-\overline{g}(\overline{\phi}X,PZ)v(Y)-2\overline{g}(\overline{\phi}X,Y)v(PZ)\}=0.	
	\end{align}
	Letting $X=\xi$ in (\ref{p106}) leads to 
	\begin{align}\label{p107}
		&(\xi\gamma)\{g(Y,Z)-\eta(Y)\eta(Z)\}+\gamma \{-B(Y,PZ)-g(Y,PZ)\tau(\xi)-u(Y)\eta(Z)\nonumber\\
		&-\eta(Y)u(Z)+\eta(Y)\eta(Z)\tau(\xi)\}-u(Z)v(Y)+\eta(Z)B(Y,U)-\eta(Z)g(\nabla_{\xi}U,Y)\nonumber\\
		&+2u(Y)\omega(Z)+\eta(Y)\xi \omega(Z)-\eta(Y)\omega(\nabla_{\xi}PZ)+\eta(Z)v(Y)\tau(\xi)-\eta(Y)\omega(Z)\tau(\xi)\nonumber\\
		&-((c+3)/4)g(Y,PZ)-((c-1)/4)\{-\eta(Y)\eta(PZ)+u(PZ)v(Y)\nonumber\\
		&+2u(Y)v(PZ)\}=0,
	\end{align}
	for any $Y,Z\in \Gamma(TM)$. Setting $Y=Z=U$ in (\ref{p107}) gives $-\gamma B(U,U)+2\omega(U)=0$. But using (\ref{p22}) and (\ref{p91}), we have $B(U,U)=\gamma$. Thus, 
	\begin{align}\label{p110}
		-\gamma^{2}+2\omega(U)=0.
	\end{align}
	 On the other hand, putting $Y=V$ and $Z=U$ in (\ref{p107}), leads to 
	\begin{align}\label{p109}
		\xi \gamma-\gamma \tau(\xi)-\frac{1}{2}(c+3)=0.
	\end{align}
	With $Y=U$ and $Z=V$, (\ref{p107}) gives 
	\begin{align}\label{p108}
		\xi \gamma-\gamma \tau(\xi)-\frac{3}{4}(c+3)=0.
	\end{align}
	From (\ref{p108}) and (\ref{p109}), we have $c=-3$ and $\xi \gamma-\gamma \tau(\xi)=0$. Note that $Y=Z=V$ gives $\gamma B(V,V)=0$. When $C$ is symmetric, we see from (\ref{p21}) that $\omega(U)=C(\zeta,U)=C(U,\zeta)=-v(U)=0$. Thus, (\ref{p110}) gives $-\gamma^{2}=0$ or simply $\gamma=0$, which shows that $(M,g)$ is contact screen geodesic. Hence, the theorem is proved.
\end{proof}
	\noindent From Theorem \ref{main3}, the following hold.
	\begin{corollary}
		There exist no contact screen umbilic null hypersurfaces of an indefinite Sasakian space form $(\overline{M}(c\ne -3), \overline{\phi},\zeta, \eta)$ with $\zeta\in \Gamma(TM)$.
	\end{corollary}
\begin{example}
\rm {
	The null hypersurface of Example \ref{example2} is contact screen geodesic, i.e., $\gamma=0$, and
	$\mathbb{R}_{2}^{7}$ is a space of constant $\overline{\phi}$-sectional curvature $-3$.} 
\end{example}

%%%%%%%%%%%%%%%%%%%%%%%%%%%%%%%%%%%%%%%%%%%%%%		
%\section*{Acknowledgments}
%Sincere thanks to the University of KwaZulu-Natal for its supportive environment. 

\end{document}